\documentclass{amsart}

\usepackage[T1]{fontenc}
\usepackage[utf8]{inputenc}
\usepackage{amsmath,amsthm,amsfonts,amscd,amssymb,eucal,latexsym,mathrsfs}
\usepackage{stmaryrd}
\usepackage{enumerate}
\usepackage{hyperref}
\usepackage[all]{xy}
\usepackage{etoolbox}

\usepackage{tikz,tikz-cd}
\usetikzlibrary{shapes.geometric}
\usetikzlibrary{arrows}
\usepackage{tqft}

\usetikzlibrary{positioning}
\usepackage{float}
\usepackage{MnSymbol}
\usetikzlibrary{matrix}

\newcommand{\toc}{\tableofcontents}

\theoremstyle{plain}
\newtheorem{theorem}{Theorem}[section]
\newtheorem*{theorem*}{Theorem}

\newtheorem{corollary}[theorem]{Corollary}
\newtheorem*{corollary*}{Corollary}

\newtheorem{lemma}[theorem]{Lemma}

\newtheorem{proposition}[theorem]{Proposition}
\newtheorem{prop}[theorem]{Proposition}

\theoremstyle{definition}
\newtheorem{remark}[theorem]{Remark}

\newtheorem{question}[theorem]{Question}

\newtheorem{definition}[theorem]{Definition}
\newtheorem*{definition*}{Definition}

\usetikzlibrary{calc,graphs}
\usepackage{xcolor}
\usetikzlibrary{arrows,decorations.pathmorphing}

\newcommand{\acts}{\curvearrowright}

\newcommand{\p}{\varphi}

\newcommand{\e}{\varepsilon}

\newcommand{\IR}{\mathbb{R}}

\newcommand{\CI}{\mathbb{C}}

\newcommand{\IN}{\mathbb{N}}
\newcommand{\NI}{\mathbb{N}}

\newcommand{\M}{\mathrm{M}}

\renewcommand{\1}{\mathbf{1}}

\DeclareMathOperator{\ssi}{\Leftrightarrow}
\DeclareMathOperator{\impl}{\Rightarrow}

\DeclareMathOperator{\Id}{\mathrm{Id}}

\newcommand{\IE}{\mathbb{E}}

\newcommand{\GL}{\mathrm{GL}}

\newcommand{\C}{\mathbf{C}}

\newcommand{\ts}{\textsection}

\newcommand{\ip}[1]{\langle#1\rangle} 

\newcommand{\St}{\mathop{\mathrm{St}}}

\newcommand{\hh}{\mathcal{H}}
\newcommand{\hk}{\mathcal{K}}

\newcommand\opnorm[1]{\|#1\|}

\title{Positive definite maps on amenable groups} 

\author{Mika\"el Pichot}
\author{Erik S\'eguin}
\address{Mika\"el Pichot, McGill University, 805 Sherbrooke St W., Montr\'eal, QC H3A 0B9, Canada}\email{mikael.pichot@mcgill.ca}
\address{Erik S\'eguin, McGill University, 805 Sherbrooke St W., Montr\'eal, QC H3A 0B9, Canada}\email{erik.seguin@mail.mcgill.ca}

\begin{document}

\begin{abstract}
We describe conditions that characterize amenability for groups in terms of positive definite functions valued in a von Neumann algebra. 
\end{abstract}

\maketitle

\section{Introduction}

Let $G$ be a (discrete)  group, $\hh$ be a  Hilbert space,  $B(\hh)$ be the algebra of bounded operators on $\hh$ endowed with the operator norm $\|\cdot\|$, and $U(\hh)$ be the group of unitary operators on $\hh$. 
  Let $\e\geq 0$. 
   A map $\varphi\colon G\to B(\hh)$ 
  is called an \emph{$\e$-representation} of $G$ if 
\[
\|\p(xy)-\p(x)\p(y)\|\leq \e
\] 
for all $x,y\in G$; we also use the term \emph{$\e$-multiplicative} to refer to such maps.

We are interested in the following question.

\begin{question}\label{Q - stable} Does there exist a non-amenable group $G$ such that
for every ${\delta>0}$ there exists some ${\e>0}$ such that for every Hilbert space $\hh$ and every $\e$-representation ${\p\colon G\to U(\hh)}$ there exists a unitary representation ${\pi\colon G\to U(\hh)}$ such that 
\[\opnorm{\p(x)-\pi(x)}<\delta\]
for all ${x\in G}$?
\end{question}

The conclusion, that there must exist a unitary representation $\pi$ approximating $\p$ to within $\delta$, is a ``stability condition'' in the sense of Ulam (see \cite{Ulam}) for the unitary representations of $G$. It is a well-known theorem of Kazhdan (see \cite[Theorem 1]{kazhdan82}) that discrete amenable groups are stable in this sense, but the converse remains open (see e.g., \cite{bot13, dc18, dcot19}). 

In the present paper, we give new characterizations of amenability which are ``intermediate'' between the usual definition of amenability (in terms of the existence of an invariant mean) and this stability condition. Roughly speaking, amenability can be viewed as a property ensuring the existence of  positive definitive maps from $G$ to $B(\hh)$ with certain properties. We refer the reader to \cite{Greenleaf} or \cite{pier84} for an introduction to amenable groups. Several well-known characterizations can be found in \cite[Theorem 2.6.8]{BO}.

Stability conditions for the unitary representations of a group can be reformulated in terms of unital positive definite maps. Let us note  that it is equivalent to prove the existence of a unitary representation $\pi$ close to a given $\e$-multiplicative map $\p$, or that of a  unital positive definite  map $\psi$ close to $\p$ which is also $\delta$-multiplicative, provided that  $\delta>0$ can be chosen to be sufficiently small compared to $\e$. The following proposition (discussed further in \ts\ref{S - stable unitary rep}) makes this statement more precise:

\begin{proposition}\label{P - Prop 1}  Let $G$ be a group and $\hh$ be a Hilbert space.
The following are equivalent:
\begin{enumerate}
\item  there exist real numbers $\kappa>0$ and $0<\delta<1$ such that for every ${0<\e<\delta}$ and every unitary $\e$-representation ${\p\colon G\to U(\hh)}$  there exists a unitary representation ${\pi\colon G\to U(\hh)}$ such that  
\[
\|\p(x)-\pi(x)\|\leq\kappa\e
\]
for all $x\in G$;
\item there exist real numbers $\kappa_1,\kappa_2>0$, $0<\delta<1$, and $p>1$ such that for every $0<\e<1$ such that $\kappa_1\e^{p-1}\leq\delta^{p-1}$   and every $\e$-representation ${\p\colon G\to U(\hh)}$  there exists a unital positive definite ${\kappa_1\e^p}$-representation ${\psi:G\to B(\hh)}$ such that 
\[\opnorm{\p(x)-\psi(x)}\leq\kappa_2\e\]
for all ${x\in G}$.
\end{enumerate}

\end{proposition}

A map ${\p\colon G\to B(\hh)}$ is said to be \emph{uniformly bounded} if 
\[
{\|{\p}\|:=\sup_{x\in G}\,\|{\p(x)}\|<\infty}.
\]
We let ${\ell^\infty(G,B(\hh))}$ denote the complex vector space of uniformly bounded maps from $G$ to $B(\hh)$. 

A version of Prop.\ \ref{P - Prop 1} for uniformly bounded representations is discussed in \ts \ref{S - ub representations}. These arguments can be use to establish stability (see e.g.\ \cite{bot13} and \cite{shtern13}). Unital positive definite maps arise as follows. 

If the group $G$ is amenable, then every invariant mean $\ell^\infty(G)$ can be extended in a natural way to an invariant mean on $\ell^\infty(G,B(\hh))$. It has been shown by Chiffre, Ozawa, and Thom in \cite[\ts 2]{dcot19} that if $\p\colon G\to B(\hh)$ is a uniformly bounded map and $\IE$ is such a mean, then the map $\psi\colon G\to B(\hh)$ defined by
\[
\psi(x) := \IE_y\p(xy)\p(y)^*, \ \ \ \  x\in G
\]
is positive definite  (see also \cite{shtern99,gh17}).

The following is our main result:

\begin{theorem}\label{T - main}
Let $G$ be a  group and $\hh$ be a Hilbert space. The following are equivalent: 
\begin{enumerate}
\item[(A)] $G$ is amenable;
\item[(B)]  there exists a $B(\hh)$-valued positive semi-definite left-sesquilinear form ${\ip{\cdot,\cdot}}$ on ${\ell^\infty(G, B(\hh))}$ and a positive real number ${\kappa>0}$ such that 
\begin{enumerate}
\item ${\ip{\1,\1}\neq0}$, where $\1\in\ell^\infty(G, B(\hh))$ refers to the constant map equal to $\Id_\hh$ for all $x\in G$;
\item ${\|{\ip{\p,\p}}\|\leq\kappa\,\|{\p}\|^2}$ for all ${\p\in\ell^\infty(G, B(\hh))}$;
\item ${\ip{\lambda(\,\cdot\,)\,\p,\p}}$ is positive definite for all ${\p\in\ell^\infty(G, B(\hh))}$;
\end{enumerate}
 \item[(C)] there exists a state ${\tau\in\mathrm{St}(\ell^\infty(G))}$ such that for every every uniformly bounded map ${\p:G\to B(\hh)}$, there exists a positive definite map ${\psi:G\to B(\hh)}$ such that 
\[
{\Phi(\p(x)^*\,\psi(x))=\tau_y\,\Phi(\p(x)^*\,\p(xy)\,\p(y)^*)}
\]
for all ${x\in G}$ and every normal linear functional ${\Phi\in B(\hh)_*}$.
\end{enumerate}
\end{theorem}

\begin{remark}
If $\p$ takes values in a von Neumann algebra $M$, then the map $\psi$ in Condition (C) can be taken to also map into $M$. 

\end{remark}

Conditions (B) and (C) describe ways to construct a positive definite map $\psi$ from the given uniformly bounded map $\p$; the construction method is explicit (although it depends on the axiom of choice in general), through the form $\ip{\cdot,\cdot}$ in (B) or the state $\tau$ in (C). By contrast, the stability condition does not yield a method for constructing the map $\psi$ from $\p$; it only requires that a $\psi$ exists which is close to $\p$ in norm.

Observe that the equality   
\[
{\Phi(\p(x)^*\,\psi(x))=\tau_y\,\Phi(\p(x)^*\,\p(xy)\,\p(y)^*)},\qquad\forall x\in G,\,\Phi\in B(\hh)_*
\]
in Condition (C) implies that $\psi$ is close to $\p$ if the latter is almost multiplicative. However, it is more precise, because it determines $\psi$ uniquely from $\p$ and $\tau$ in many cases. In a sense, (B) and (C) can be viewed as characterizations of the invariance of a mean on $G$ in terms of positive definite maps. 

Accordingly, one can introduce the notion of a ``partially invariant mean'' by requiring that Condition (C) holds only for some subset of ${\ell^\infty(G,B(\hh))}$ for some state $\tau$. For example, the stability condition applies only to almost multiplicative maps and would correspond to a partial invariance condition in this sense.

In \ts\ref{S - restricting to epsilon rep}, we discuss this sort of partial invariance by  restricting Condition (C) to almost multiplicative maps, and deduce the following from the proof of Kazhdan's theorem and the stability theorems established in  \cite{bot13}.

\begin{proposition}\label{P - nonamenable intro}
Suppose that $G$ contains a non-abelian free group. For every state $\tau$ on  $\ell^\infty(G)$, every Hilbert space $\hh$, and small enough $\e>0$, there exists an $\e$-representation ${\p: G\to U(\hh)}$ such that for every positive definite map ${\psi: G\to B(\hh)}$, there exists some normal linear functional ${\Phi\in B(\hh)^*}$ such that 
\[
{\Phi(\p(x)^*\,\psi(x))\neq\tau_y\,\Phi(\p(x)^*\,\p(xy)\,\p(y)^*)}
\]
for some ${x\in G}$.  
\end{proposition}

We conclude the paper with a remark on unitarily invariant norms. While the
 classical stability theorems   \cite{GKR,HK,kazhdan82} focus on the operator norm,  several interesting results have been established in more general situations \cite{GKR,gh17,dcot19}, in particular with respect to ultraweakly lower semi-continuous unitary invariant norms. One can deduce (see Prop.\ \ref{P - unitarily invariant norms})  estimates for such norms from Condition (C).    

\bigskip
\bigskip

\noindent\textbf{Acknowledgment.}  The work presented in this paper is part of the M.Sc.\ thesis of the second author. The authors are partially funded by NSERC Discovery Fund 234313. 

\toc

\section{$\delta$-unitary maps on groups}

Let $G$ be a group and $\hh$ be a Hilbert space. 

\begin{definition} Let $\delta\geq 0$.
A map $\p:G\to B(\hh)$ is said to be \emph{$\delta$-unitary} if $\opnorm{\Id_\hh-\p(x)\,\p(x)^*}\leq \delta$ and $\opnorm{\Id_\hh-\p(x)^*\,\p(x)}\leq \delta$ for every $x\in G$.
\end{definition}

The present section contains basic results on $\delta$-unitary maps, in particular in relation with positive definite mappings, representations, and $\e$-representations. While the techniques and computations are elementary and well-known, we provide complete proofs for the convenience of the reader. We recall that a map $\p\colon G\to B(\hh)$ is said to be positive definite if for every finite set $\{x_1,\ldots, x_n\} \subset G$, the operator matrix $(\p(x_i^{-1}x_j))_{ij}$ is positive.

\begin{lemma}\label{P - upd epsilon multiplicative}
Let $\delta\geq 0$ and suppose $\p:G\to B(\hh)$ is unital positive definite. The following are equivalent:
\begin{enumerate}
\item $\p$ is $\delta$-unitary.
\item $\p$ is $\delta$-multiplicative.
\end{enumerate}
\end{lemma}

\begin{proof}
(2) $\impl$ (1): we have
\[
\opnorm{\Id_\hh-\p(x)\,\p(x)^*}=\opnorm{\p(e)-\p(x)\,\p(x^{-1})}\leq \delta
\]
for every $x\in G$, and similarly, $\opnorm{\Id_\hh-\p(x)^*\,\p(x)}\leq \delta$.

(1) $\impl$ (2): This follows from  Lemma \ref{stinespring-ineq}.
\end{proof}

\begin{lemma}\label{stinespring-ineq}
If ${\p:G\to B(\hh)}$ is a positive definite map such that ${\opnorm{\p(e)}\leq1}$, then 
\[\opnorm{\p(xy)-\p(x)\,\p(y)}\leq\opnorm{\p(e)-\p(x)\,\p(x)^*}^{1/2}\,\opnorm{\p(e)-\p(y)^*\,\p(y)}^{1/2}\]
for all ${x,y\in G}$. 
\end{lemma}

\begin{proof}
It follows by Stinespring's dilation theorem that there exists a Hilbert space $\hk$, a bounded operator ${U:\hh\to\hk}$, and a unitary representation ${\pi\colon G\to U(\hk)}$ such that ${\p(x)=U^*\,\pi(x)\,U}$ for all ${x\in G}$. As ${UU^*}$ is a positive operator such that 
\[\opnorm{UU^*}=\opnorm{U^*}^2=\opnorm{U}^2=\opnorm{\p(e)}\leq1\]
it follows that ${\Id_\hk-UU^*}$ is a positive operator, and thus ${(\Id_\hk-UU^*)^{1/2}}$ is well-defined, so 
\begin{align*}
\opnorm{\p(xy)-\p(x)\,\p(y)}&=\opnorm{U^*\,\pi(xy)\,U-U^*\,\pi(x)\,UU^*\,\pi(y)\,U}
\\&=\opnorm{U^*\,\pi(x)\,(\Id_\hk-UU^*)^{1/2}\,(\Id_\hk-UU^*)^{1/2}\,\pi(y)\,U}
\\&\leq\opnorm{U^*\,\pi(x)\,(\Id_\hk-UU^*)^{1/2}}\,\opnorm{(\Id_\hk-UU^*)^{1/2}\,\pi(y)\,U}
\\&=\opnorm{(\Id_\hk-UU^*)^{1/2}\,\pi(x)^*\,U}\,\opnorm{(\Id_\hk-UU^*)^{1/2}\,\pi(y)\,U}
\\&=\opnorm{U^*\,\pi(x)\,(\Id_\hk-UU^*)\,\pi(x)^*\,U}^{1/2}\,\opnorm{U^*\,\pi(y)^*\,(\Id_\hk-UU^*)\,\pi(y)\,U}^{1/2}
\\&=\opnorm{U^*U-U^*\,\pi(x)\,UU^*\,\pi(x)^*\,U}^{1/2}\,\opnorm{U^*U-U^*\,\pi(y)^*\,UU^*\,\pi(y)\,U}^{1/2}
\\&=\opnorm{\p(e)-\p(x)\,\p(x)^*}^{1/2}\,\opnorm{\p(e)-\p(y)^*\,\p(y)}^{1/2}
\end{align*}
for all ${x,y\in G}$. 
\end{proof}

We shall use the following terminology for convenience.
We say that a bounded operator $T\colon \hh\to \hk$ between Hilbert spaces is \emph{$\delta$-isometric}  if $\|\Id_\hh-T^*T\|\leq \delta$, and \emph{$\delta$-unitary} if $T$ and $T^*$ are $\delta$-isometric; a map $\p\colon G\to B(\hh)$ is $\delta$-isometric if $\p(x)$ is $\delta$-isometric for every $x\in G$. We define  $\p^*\colon G\to B(\hh)$ by $\p^*(x):=\p(x)^*$ for all $x\in G$; clearly, $\p$ is $\delta$-unitary if and only if $\p$ and $\p^*$ are $\delta$-isometric if and only if $\p(x)$ is $\delta$-unitary for every $x\in G$; if $\p$ is positive definite, then $\p^*(x)=\p(x^{-1})$ for all $x\in G$.

\begin{lemma}
\label{P - upd epsilon isometric}
Let $\delta\geq 0$ and suppose $\p:G\to B(\hh)$ is  positive definite. The following are equivalent:
\begin{enumerate}
\item $\p$ is $\delta$-unitary.
\item $\p$ is $\delta$-isometric.
\end{enumerate}
\end{lemma}

\begin{lemma}\label{L - invertible eps isom}
Let $\delta\geq 0$ and suppose that $T\colon \hh\to \hk$ is a bounded operator between Hilbert spaces. Then
\begin{enumerate}
\item If $T$ is $\delta$-isometric, then $1-\delta\leq \|T\|^2\leq 1+\delta$.
\item If $T$ is invertible, $\|T\|^2\leq 1+\delta$, and $\|T^{-1}\|^2\leq 1+\delta$, then $T$ is $\delta$-unitary.
\end{enumerate}
\end{lemma}

\begin{proof}
(1) We have
\begin{align*}
\|T\|^2=\|T^*T\| = \|\Id_\hh- (\Id_\hh-T^*T)\|&\leq 1+\delta\\
&\geq 1-\delta.
\end{align*}
(2) For every unit vector $\xi\in \hh$ we have
\[
1=\|T^{-1}(T^{-1})^*T^*T\xi\|\leq (1+\delta)\|T^*T\xi\|\leq (1+\delta)^2
\]
therefore $\|T^*T\xi\|\in [(1+\delta)^{-1},1+\delta]\subset[1-\delta,1+\delta]$. This implies that $\mathrm{Sp}(\Id_\hh-T^*T)\subset [-\delta,\delta]$, so $T$ is $\delta$-isometric. Exchanging the role of $T$ and $T^*$, it follows identically that $T^*$ is $\delta$-isometric. Therefore, $T$ is $\delta$-unitary.
\end{proof}

\begin{lemma}
\label{P - upd epsilon isometric}
Let $\delta\geq 0$ and suppose $\pi\colon G\to B(\hh)$ is a uniformly bounded representation of $G$ on $\hh$. The following are equivalent:
\begin{enumerate}
\item $\|\pi\|^2\leq 1+\delta$.
\item $\pi$ is $\delta$-unitary.
\item $\pi$ is $\delta$-isometric.
\end{enumerate}
\end{lemma}

\begin{proof}
This is a direct consequence of Lemma \ref{L - invertible eps isom}.
\end{proof}

For the remainder of this section we fix nonnegative real numbers $\e,\delta,\eta\geq 0$ and bounded maps ${\p,\,\psi:G\to B(\hh)}$.

\begin{lemma}\label{L - isom-prox} If $\p$ is $\delta$-isometric and $\|\p-\psi\|\leq\eta$, then $\psi$ is ${(\delta+(\|\p\|+\|\psi\|)\eta)}$-isometric. 
\end{lemma}

\begin{proof}
We have
\begin{align*}
\opnorm{\Id_\hh-\psi(x)^*\,\psi(x)}&\leq\opnorm{\Id_\hh-\p(x)^*\,\p(x)}+\opnorm{\p(x)^*\,\p(x)-\p(x)^*\,\psi(x)}\\
&\ \ \ \ \ \ \ \ \ +\opnorm{\p(x)^*\,\psi(x)-\psi(x)^*\,\psi(x)}
\\&\leq\opnorm{\p(x)^*\,\p(x)-\p(x)^*\,\psi(x)}+\opnorm{\p(x)^*\,\psi(x)-\psi(x)^*\,\psi(x)}+\delta
\\&\leq\opnorm{\p(x)^*}\,\opnorm{\p(x)-\psi(x)}+\opnorm{\p(x)^*-\psi(x)^*}\,\opnorm{\psi(x)}+\delta
\\&=\opnorm{\p(x)}\,\opnorm{\p(x)-\psi(x)}+\opnorm{\p(x)-\psi(x)}\,\opnorm{\psi(x)}+\delta
\\&\leq\|{\p}\|\,\opnorm{\p(x)-\psi(x)}+\opnorm{\p(x)-\psi(x)}\,\|{\psi}\|+\delta
\\&\leq\left(\|{\p}\|+\|{\psi}\|\right)\eta+\delta
\end{align*}
for all ${x\in G}$. 
\end{proof}

\begin{lemma}\label{L - unitary-prox}
If $\p$ is $\delta$-unitary and $\|\p-\psi\|\leq\e$, then $\psi$ is ${(
\delta+(\|\p\|+\|\psi\|)\e)}$-unitary.
\end{lemma}

\begin{proof} This follows from Lemma \ref{L - isom-prox} and the fact that $\|\p^*-\psi^*\|=\|\p-\psi\|$. 
\end{proof}

\begin{lemma}\label{double-almost-prox}
If $\p$ is a $\delta$-unitary $\e$-representation  and $\|\p-\psi\|\leq\eta$, then $\psi$ is a ${(\delta+(\|\p\|+\|\psi\|)\,\eta)}$-unitary ${(\e+(1+\|\p\|+\|\psi\|)\,\eta)}$-representation. 
\end{lemma}

\begin{proof}
By Lemma \ref{L - unitary-prox}, $\psi$ is ${(\delta+(\|\p\|+\|\psi\|)\,\eta)}$-unitary. Furthermore, it follows that 
\begin{align*}
\opnorm{\psi(xy)-\psi(x)\,\psi(y)}&\leq\opnorm{\psi(xy)-\p(xy)}+\opnorm{\p(xy)-\p(x)\,\p(y)}\\
&\ \ \ \ \ \ \ \ \ +\opnorm{\p(x)\,\p(y)-\psi(x)\,\psi(y)}
\\&<\opnorm{\p(x)\,\p(y)-\psi(x)\,\psi(y)}+\e+\eta
\\&\leq\opnorm{\p(x)\,\p(y)-\p(x)\,\psi(y)}+\opnorm{\p(x)\,\psi(y)-\psi(x)\,\psi(y)}+\e+\eta
\\&\leq\opnorm{\p(x)}\,\opnorm{\p(y)-\psi(y)}+\opnorm{\p(x)-\psi(x)}\,\opnorm{\psi(y)}+\e+\eta
\\&\leq\|\p\|\,\opnorm{\p(y)-\psi(y)}+\opnorm{\p(x)-\psi(x)}\,\|\psi\|+\e+\eta
\\&\leq\left(1+\|\p\|+\|\psi\|\right)\eta+\e
\end{align*}
for all ${x,y\in G}$, and thus $\psi$ is ${(\e+(1+\|\p\|+\|\psi\|)\,\eta))}$-multiplicative. 
\end{proof}

If ${\p,\psi:G\to B(\hh)}$ we write ${\psi\prec\p}$ if ${\psi(s)}$ belongs to the von Neumann algebra generated by ${\p(G)}$ for all ${s\in G}$. This is a preorder on ${\ell^\infty(G,B(\hh))}$.

We  shall require the following lemma (compare the proof  Theorem 3.1 of \cite{dcot19}).

\begin{lemma}\label{L - invertible-unitary-almost-rep-approx} If ${\p\colon G\to \GL(\hh)}$ is a $\delta$-unitary $\e$-representation, then there exists an ${(\e+(1+\|\p\|+\|\psi\|)\,\delta)}$-representation ${\psi:G\to U(\hh)}$  such that 
${\|{\p-\psi}\|\leq\delta}$.
Furthermore, ${\psi\prec\p}$. 
\end{lemma}

\begin{proof}
 Let ${\p(x)=U_x\,|{\p(x)}|}$ be the polar decomposition of ${\p(x)}$; as ${\p(x)}$ is invertible, ${U_x}$ can be taken to be unitary. Let ${\psi:G\to U(\hh)}$ be the unitary map defined by ${\psi(x)=U_x}$. As ${|{\p(x)}|}$ is a positive operator for all ${x\in G}$ and ${|{1-t}|\leq|{1-t^2}|}$ for all ${t\in[0,\infty)}$,  we have 
\[
\opnorm{\Id_\hh-|{\p(x)}|}\leq\opnorm{\Id_\hh-|{\p(x)}|^2}
\]
 therefore  
\begin{align*}
\opnorm{\psi(x)-\p(x)}&=\opnorm{\psi(x)-\psi(x)\,|{\p(x)}|}
\\&=\opnorm{\Id_\hh-|{\p(x)}|}
\\&\leq\opnorm{\Id_\hh-|{\p(x)}|^2}
\\&=\opnorm{\Id_\hh-\p(x)^*\,\p(x)}\leq\delta
\end{align*}
for all ${x\in G}$. It follows by Lemma \ref{double-almost-prox} that $\psi$ is an ${(\e+(1+\|\p\|+\|\psi\|)\delta)}$-representation. 
\end{proof}

We also remark (following \cite{dcot19}) that the inequality 
\[
\opnorm{\Id_\hh-|{\p(x)}|}\leq\opnorm{\Id_\hh-|{\p(x)}|^2},
\] 
used in the previous lemma for the operator norm, holds in fact for an arbitrary unitarily invariant norm:

\begin{lemma}\label{square-ineq}
Let $\hh$ be a separable Hilbert space and $\|\cdot\|$ be a unitarily invariant norm on ${B(\hh)}$. If ${T\in B(\hh)}$ is a positive operator, then 
\[
{\|{\Id_\hh-T}\|\leq\|{\Id_\hh-T^2}}\|
\]
\end{lemma}

\begin{proof}
As ${|{1-t}|\leq|{1-t^2}}|$ for all ${t\in[0,\infty)}$, it follows   that 
\[
{\|{\Id_\hh-T}\|=\|{|{\Id_\hh-T}}|\|\leq\|{|{\Id_\hh-T^2}}|\|=\|{\Id_\hh-T^2}}\|
\]
by Proposition 2.6 in \cite{dcot19}.
\end{proof}

\begin{lemma}\label{unitary-almost-rep-approx} Suppose $\delta<1$. If ${\p\colon G\to B(\hh)}$ is a $\delta$-unitary $\e$-representation, then there exists an ${(\e+4\delta)}$-representation ${\psi:G\to U(\hh)}$  such that 
${\|{\p-\psi}\|\leq\delta}$.
Furthermore, ${\psi\prec\p}$. 
\end{lemma}

\begin{proof}
As $\p$ is $\delta$-unitary and ${\delta<1}$, we have 
\[{\opnorm{\Id_\hh-\p(x)^*\,\p(x)}<1,\qquad\opnorm{\Id_\hh-\p(x)\,\p(x)^*}<1}\]
for all ${x\in G}$, so ${\p(x)^*\,\p(x)}$ and ${\p(x)\,\p(x)^*}$ are invertible for all ${x\in G}$, and thus in turn ${\p(x)}$ and ${\p(x)^*}$ are invertible for all ${x\in G}$.  As 
\[
{1+\|\p\|+\|\psi\|=2+\|\p\|\leq2+\sqrt{1+\delta}<2+\sqrt{2}},
\] 
it follows by Lemma \ref{L - invertible-unitary-almost-rep-approx} that $\psi$ is an ${(\e+4\delta)}$-representation. 
\end{proof}

\section{Stability and positive definite maps}\label{S - stable unitary rep}

Let $G$ be a  group. Consider the following stability condition, which depends on a positive real number $\kappa>0$ and a Hilbert space $\hh$:

\begin{quote}
there exists a real number $0<\delta<1$ such that for every ${0<\e<\delta}$ and every unitary $\e$-representation ${\p\colon G\to U(\hh)}$  there exists a unitary representation ${\pi\colon G\to U(\hh)}$ such that  $\|\p-\pi\|\leq\kappa\e$.
\end{quote}

\noindent Let us say that $G$ is \emph{stable} if this condition holds for some $\kappa$ and all $\hh$. The theorem of Kazhdan \cite{kazhdan82} mentioned in the introduction states:

\begin{theorem}[Kazhdan, 1982]
Every  amenable group is stable (with constant $\kappa=2$).
\end{theorem}

(Here and below, by amenable group, we mean a group which is amenable as a discrete group.)

In the present section, we note that the stability condition is equivalent to every such map $p$ being approximated by a positive definite almost unitary map defined over the same Hilbert space, rather than a unitary representation.  

In the following proposition $\hh$ is fixed, while the real numbers $\kappa,\kappa_1,\kappa_2,p,\delta$ may vary.

\begin{proposition}\label{stability-equiv}
Let $G$ be a group and  $\hh$ be a Hilbert space. The following are equivalent: 
\begin{enumerate}
\item there exists real numbers $\kappa>0$ and $0<\delta<1$ such that for every ${0<\e<\delta}$ and every unitary $\e$-representation ${\p\colon G\to U(\hh)}$  there exists a unitary representation ${\pi\colon G\to U(\hh)}$ such that  
\[
\|\p(x)-\pi(x)\|\leq\kappa\e
\]
for every $x\in G$;

\item 
there exist real numbers $\kappa_1,\kappa_2>0$, $0<\delta<1$, and $p>1$ such that for every $0<\e<1$ such that $\kappa_1\e^{p-1}\leq\delta^{p-1}$  and every $\e$-representation ${\p\colon G\to U(\hh)}$  there exists an ${\kappa_1\e^p}$-representation  ${\psi:G\to U(\hh)}$ such that 
\[\opnorm{\p(x)-\psi(x)}\leq\kappa_2\e\]
for all ${x\in G}$;

\item there exist real numbers  $\kappa_1,\kappa_2>0$, $0<\delta<1$, and $p>1$ such that for every $0<\e<1$ such that $\kappa_1\e^{p-1}\leq\delta^{p-1}$  and every $\e$-representation ${\p\colon G\to U(\hh)}$  there exists a unital positive definite ${\kappa_1\e^p}$-representation ${\psi:G\to B(\hh)}$ such that 
\[\opnorm{\p(x)-\psi(x)}\leq\kappa_2\e\]
for all ${x\in G}$;

\item there exist real numbers  $\kappa_1,\kappa_2>0$, $0<\delta<1$, and $p>1$ such that for every $0<\e<1$ such that $\kappa_1\e^{p-1}\leq\delta^{p-1}$  and every $\e$-representation ${\p\colon G\to U(\hh)}$  there exists a unital positive definite ${\kappa_1\e^p}$-unitary map ${\psi:G\to B(\hh)}$ such that 
\[\opnorm{\p(x)-\psi(x)}\leq\kappa_2\e\]
for all ${x\in G}$.

\end{enumerate}
\end{proposition}

\begin{proof} The implications (1) $\impl$ (2), (3), (4) are clear, and  the equivalence (3) $\ssi$ (4) follows immediately from Lemma \ref{P - upd epsilon multiplicative}. 

(3) $\impl$ (2)
Let $\kappa_1,\kappa_2>0$, $0<\delta<1$, and $p>1$, and suppose that ${\psi:G\to B(\hh)}$ is a map witnessing ${(3)}$. Then $\psi$ is ${\kappa_1\e^p}$-unitary and as ${\kappa_1\e^p\leq\delta^{p-1}\e<1}$, it follows from Lemma \ref{unitary-almost-rep-approx} that there exists a ${5\kappa_1\e^p}$-representation ${\pi:G\to U(\hh)}$ such that 
\[\opnorm{\p(x)-\pi(x)}\leq\opnorm{\p(x)-\psi(x)}+\opnorm{\psi(x)-\pi(x)}<(\kappa_2+ \kappa_1\e^{p-1})\e<(1+\kappa_2)\e\]
for all ${x\in G}$. Therefore ${(3)\Rightarrow(2)}$ holds. 

(2) $\impl$ (1) Let ${\p_0=\p}$, let ${\e_0=\e}$, and for every ${n\in\NI}$, let ${\p_n}$ be the map yielded by applying ${(2)}$ to ${\p_{n-1}}$; then ${\p_n}$ is a unitary ${\e_n}$-representation, where 
\[
\e_n=\kappa_1\e_{n-1}^p=\kappa_1^{1+p+\cdots+p^{n-1}}\e^{p^n}=\kappa_1^{(p-1)^{-1}(p^n-1)}\e^{p^n}=\kappa_1^{-(p-1)^{-1}}\big(\kappa_1^{(p-1)^{-1}}\e\big)^{p^n-1}\,\e
\] 
for all $n\in \IN$.
As 
\[\opnorm{\p_n(x)-\p_{n+1}(x)}\leq\kappa_2\e_n\]
for all ${x\in G}$ and ${n\in\NI}$, it follows that ${(\p_n(x))}$ converges for all ${x\in G}$. Let ${\pi\colon G\to U(\hh)}$ be the unitary map defined by 
\[\pi(x)=\lim\p_n(x)\]
Then $\pi$ is a representation and 
\begin{align*}
\opnorm{\p(x)-\pi(x)}&\leq\opnorm{\p(x)-\p_1(x)}+\sum_{n\,=\,1}^\infty\opnorm{\p_n(x)-\p_{n+1}(x)}
\\&\leq\Big(1+\kappa_1^{-(p-1)^{-1}}\sum_{n\,=\,1}^\infty\big(\kappa_1^{(p-1)^{-1}}\e\big)^{p^n-1}\Big)\,\kappa_2\mskip0.25\thinmuskip\e
\\&\leq\Big(1+\kappa_1^{-(p-1)^{-1}}\sum_{n\,=\,1}^\infty\delta^{p^n-1}\Big)\,\kappa_2\mskip0.25\thinmuskip\e
\end{align*}
for all ${x\in G}$, which proves the implication ${(2)\Rightarrow(1)}$. 
\end{proof}

\begin{remark}  
 It is easy to show that the equivalences in Proposition \ref{stability-equiv} still hold if 
  we assume ${\psi\prec\p}$ in the last three conditions and ${\pi\prec\p}$ in the first.
\end{remark}

\section{Remarks on uniformly bounded representations}\label{S - ub representations}

 Let $\p\colon G\to U(\hh)$ be an $\e$-representation, $\e>0$. We can establish stability in two steps as follows: 
\begin{enumerate}
\item first take $\p$ to a uniformly bounded representation $\psi\colon G\to B(\hh)$ (fixing multiplicativity), and
\item then take $\psi\colon G\to B(\hh)$ to a unitary representation $\pi\colon G\to U(\hh)$ (fixing unitarity)
\end{enumerate}
in such a way that $\|\p-\psi\|$ and $\|\psi-\pi\|$ are small. 

In Prop.\ \ref{stability-equiv}, these two steps are performed simultaneously, following the idea of \cite{bot13} (cf.\ \cite{kazhdan82,johnson88,shtern99}).

Observe that the first step may result in a map $\psi$ such that ${\|{\psi}\|>\|{\p}}\|$; one is therefore lead to consider the following form of stability in the case of uniformly bounded representations:

\begin{quote}
there exist positive real number $\kappa, M\geq 1$ and ${0<\delta <1}$ such that for every ${0<\e\leq\delta}$ and every unital uniformly bounded $\e$-representation ${\p:G\to B(\hh)}$  there exists a representation ${\psi:G\to B(\hh)}$ such that ${\|{\psi}\|\leq M\|{\p}}\|$ and 
\[
{\|{\p(x)-\psi(x)}\|\leq\kappa\,\|{\p}\|\,\e}
\]
for all ${x\in G}$.
\end{quote}

We have the following implication:

\begin{prop}\label{stability-equiv-2} Let $G$ be a group and $\hh$ be a Hilbert space. 
 Suppose that there exist positive real numbers ${\kappa_1,\kappa_2\geq 1}$, ${0<\delta<1}$, and ${p>1}$, and a non-decreasing function $F:(0,1)\to[1,\infty)$ such that 
\[
M:=\prod_{n\,=\,0}^\infty F(\delta^{p^n})<\infty
\]
such that for every $\e>0$ such that $\kappa_1\e^{p-1}\leq\delta^{p-1}$ and every unital uniformly bounded $\e$-representation ${\p:G\to B(\hh)}$  there exists a unital ${\kappa_1\e^p}$-representation ${\psi:G\to B(\hh)}$ such that ${\|{\psi}
\|\leq F(\e)\,\|{\p}\|}$ and 
\[
{\|{\p-\psi}\|\leq\kappa_2\|{\p}\|\,\e}.
\]
Then $G$ satisfies the aforementioned stability condition for uniformly bounded representations. 
\end{prop}

\begin{proof}
Let ${\psi_0=\p}$, let ${\e_0=\e}$, and for every ${n\in\IN}$, let ${\psi_n}$ be the map yielded by applying the condition to ${\psi_{n-1}}$; then ${\psi_n}$ is a uniformly bounded unital ${\e_n}$-representation where
\[
\e_n=\kappa_1\e_{n-1}^p=\kappa_1^{1+p+\cdots+p^{n-1}}\e^{p^n}=\kappa_1^{(p-1)^{-1}(p^n-1)}\e^{p^n}=\kappa_1^{-(p-1)^{-1}}\big(\kappa_1^{(p-1)^{-1}}\e\big)^{p^n}\leq \delta^{p^n}
\]
and  
\[
{\|{\psi_n}\|\leq F(\e_{n-1})\,\|{\psi_{n-1}}\|\leq\cdots\leq\|{\p}\|\prod_{j\,=\,0}^{n-1}F(\e_j)\leq\|{\p}\|\prod_{j\,=\,0}^{n-1} F(\delta^{p^j})\leq M\,\|{\p}\|}
\]
for all ${n\in\IN}$. In turn, this implies that 
\[
{\opnorm{\psi_n(x)-\psi_{n+1}(x)}\leq\kappa_2\,\|{\psi_n}\|\e_n\leq M\kappa_2\,\|{\p}\|\e_n}
\]
for all ${x\in G}$ and ${n\in\IN}$, and thus it follows that ${(\psi_n(x))}$ converges for all ${x\in G}$. Let ${\pi:G\to B(\hh)}$ be the unital uniformly bounded map defined by 
\[
{\psi(x)=\lim\psi_n(x)}
\]
Then $\psi$ is a representation and 
\begin{align*}
\|{\p(x)-\psi(x)}\|&\leq\|{\p(x)-\psi_1(x)}\|+\sum_{n\,=\,1}^\infty\|{\psi_n(x)-\psi_{n+1}(x)}\|
\\&\leq\Big(1+\sum_{n\,=\,1}^\infty\delta^{p^n-1}\Big)\,M\kappa_2\,\|{\p}\|\,\e
\end{align*}
for all ${x\in G}$. 
\end{proof}

 Suppose that $G$ is amenable. Shtern proves in \cite[Theorem 8]{shtern13}  that the assumptions in Prop.\ \ref{stability-equiv-2} are satisfied, and therefore that the stability condition (for uniformly boundes representations) holds. 
  Furthermore,  it follows by Dixmier's theorem  that if ${\psi:G\to B(\hh)}$ is a uniformly bounded representation, then there exists a unitary representation ${\pi:G\to U(\hh)}$ such that 
\[
{\|{\psi(x)-\pi(x)}\|\leq\|{\psi}\|\,\big(\|{\psi}\|^2-1\big)}
\]
for all ${x\in G}$. In particular, if $\psi$ is $\delta$-unitary, then  $\|{\psi}\|^2-1\leq \delta$; if in addition we have that  ${\|{\psi}\|\leq M\,\|{\p}\|}$ and
\[
{\|{\p(x)-\psi(x)\|}\leq\kappa\,\|{\p}\|\,\e}
\]
for all ${x\in G}$, then 
\[
{\|{\p(x)-\pi(x)\|}\leq(\kappa\e+M\delta)\|{\p}\|}
\]
for all ${x\in G}$ (compare the remarks in Section 7 of \cite{johnson88}).

\section{Semi-unitarizability}\label{S - semi-unitarizable}

Let $M$ be a von Neumann algebra acting on a Hilbert space $\hh$, and let $V$ be a normal dual Banach right $M$-module (i.e., such that the maps $x\mapsto vx$ is weak$^*$-continuous). We write $B_M(V)$ for the set of bounded $M$-module maps on $V$.

\begin{definition}\label{D - semi-unitarizable}
A uniformly bounded linear representation $\pi\colon G\to B_M(V)$ is \emph{semi-unitarizable} if there exists an $M$-valued 
positive semi-definite sesquilinear form $\ip{\,,\,}\colon V\times V\to M$ which is
\begin{enumerate}
\item continuous:
\[
\|\ip{u,\ v}\|\leq C\|u\|\|v\|\ \ \forall u,v\in V, \text{ for some constant } C
\]
\item invariant:
\[
\ip{\pi(s)u, \pi(s)v}=\ip{u,v}\ \ \forall u,v\in V,
\]
\item non-degenerate:
\[
\ip{u, u}\neq 0, \ \forall u\in V^G,\ u\neq 0.
\]
\end{enumerate}
We say that $\pi$ is \emph{unitarizable} if   $\ip{\,,\,}$ can furthermore be chosen to be positive definite, i.e.,  $\ip{u,\ u}\neq 0, \ \forall u\in V$, $u\neq 0$. 
\end{definition}

It is well-known (considering the case $M=\CI$ of trivial coefficients) that every representation of a finite group on a finite dimensional vector space  is unitarizable, and every uniformly bounded representation of an amenable group on a Hilbert space is unitarizable. The Dixmier problem asks whether this property characterizes amenable groups.
  
Consider the complex vector space ${\ell^\infty(G,M)}$ of uniformly bounded maps from $G$ to $M$, which we view as a normal Banach  $M$-bimodule  under the action defined by $(S\p T)(x)=S\p(x)T$ for every $x\in G$ and $S,T\in M$.

\begin{proposition}\label{P - semi-unitarizable} Let $G$ be a group and $M$ be a von Neumann algebra. The following are equivalent:
\begin{enumerate}
\item $G$ is amenable;
\item the left regular representation 
\begin{align*}
\lambda\colon G\acts \ell^\infty(G,M),\ \lambda_x(\p)(y):=\p(x^{-1}y),\ \p\in \ell^\infty(G,M),\ x,y\in G 
\end{align*}
is semi-unitarizable.
\end{enumerate}
\end{proposition}

We shall use the following notion of an invariant mean on the space $\ell^\infty(G,M)$.

\begin{definition}\label{D - Inv mean BH}
Let $G$ be a  group and $M$ be a von Neumann algebra. An \emph{invariant mean} on ${\ell^\infty(G,M)}$ is a linear map ${\IE:\ell^\infty(G,M)\to M}$ such that ${\opnorm{\IE}=1}$, ${\IE(T\1)=T}$ for all ${T\in M}$, and 
\[
{\IE_{y}\p(xy)=\IE_{y}\p(y)}
\]
for all ${\p\in\ell^\infty(G,M)}$ and ${x\in G}$. 
\end{definition}

\begin{lemma}
An invariant mean $\IE$ in the sense of Def.\ \ref{D - Inv mean BH} is positive and satisfies 
\[
{\IE(S\p T)=S\,\IE(\p)\,T}
\]
for all ${\p\in\ell^\infty(G,M)}$ and ${S,T\in M}$. 
\end{lemma}

\begin{proof} This is an obvious application of Tomiyama's theorem. Namely, consider the canonical inclusion map ${i\colon \M\hookrightarrow\ell^\infty(G,\M)}$. As ${(i\circ\IE)(T\1)=i(T)=T\1}$ for all ${T\in M}$ and ${\|{i\circ\IE}\|=1}$,  it follows by Theorem 1 in \cite{tomiyama57} that 
\[
{(i\circ\IE)(S\mskip0.5\thinmuskip\p\mskip0.5\thinmuskip T)=i(S)\,(i\circ\IE)(\p)\,i(T)=i(S\,\IE(\p)\,T)}
\]
for all ${S,T\in\M}$ and ${\p\in\ell^\infty(G,\M)}$; as $i$ is an embedding, this implies that 
\[
{\IE(S\mskip0.5\thinmuskip\p\mskip0.5\thinmuskip T)=S\,\IE(\p)T}
\]
for all ${S,T\in\M}$ and ${\p\in\ell^\infty(G,\M)}$. Furthermore, as ${i\circ\IE}$ is positive and $i$ is an isometric *-homomorphism, it follows that $\IE$ is positive. 
\end{proof}

Every  invariant mean $\IE\in \ell^\infty(G)^*$ extends to an invariant mean $\IE\colon \ell^\infty(G,M)\to M=(M_*)^*$ defined by
\[
\Phi(\IE_x \p(x))=\IE_x \Phi(\p(x)), \forall \Phi\in M_*
\]
(see e.g., \cite{kazhdan82, shtern99,shtern13, bot13,dcot19}). Explicitly, for each ${\p\in\ell^\infty(G,M)}$ and ${\eta\in\hh:=L^2(M)}$, define a bounded linear functional ${\psi_{\p,\eta}\in\hh^*}$ by 
\[
\psi_{\p,\eta}(\xi)=\IE_{x}\ip{\xi,\p(x)\,\eta},
\]
where $\IE$ is an invariant mean on $\ell^\infty(G)$. It follows by the Riesz representation theorem that a map ${\IE:\ell^\infty(G,M)\to B(\hh)}$ can be defined by 
\[
\psi_{\p,\eta}(\xi)=\ip{\xi,\IE(\p)(\eta)}
\]
Clearly, $\IE(\p)\in M''=M$, so $\IE$ defines an invariant mean on ${\ell^\infty(G,M)}$.

\begin{lemma}\label{L - mean-ext} The following are equivalent:
\begin{enumerate}
\item $G$ is amenable;
\item there exists an invariant mean on ${\ell^\infty(G,M)}$;
\item there exists a  bounded linear map ${\IE:\ell^\infty(G,M)\to M}$ such that  
 ${\IE(\1)\neq 0}$ and
 \[
 {\IE_{y}\p(xy)=\IE_{y}\p(y)}
 \] 
 for all ${\p\in\ell^\infty(G,M)}$ and ${x\in G}$.
 \end{enumerate} 
\end{lemma}

\begin{proof}
(1) $\impl$ (2) was noted above and (2) $\impl$ (3) is obvious. To show (3) $\impl$ (1), fix a unit vector ${\xi\in(\ker\,\IE(\1)^{1/2})^\perp}$ and let ${\psi:\ell^\infty(G,M)\to\C}$ be the linear map defined by 
\[
\psi(f)=\ip{\IE(f)\,\xi,{\xi}}.
\]
The restriction $\psi_{|\ell^\infty(G)}$ to the canonical embedding $\ell^\infty(G)\subset \ell^\infty(G,M)$   defines a nontrivial left invariant  bounded linear functional on $\ell^\infty(G)$. Therefore,  $G$ is  amenable. 
\end{proof}

\begin{remark} The invariant mean $\IE^M$ defined in (1) $\impl$ (2) satisfies 
\begin{enumerate}
\item[] ${\IE_{x}\ip{\xi,\p(x)\,\eta}=\ip{\xi,\IE_{x}^M\p(x)\,\eta}}$
\item[] ${\IE_{x}\ip{\p(x)\,\xi,\eta}=\ip{\IE_{x}^M\p(x)\,\xi,\eta}}$
\end{enumerate}
for every ${\p\in\ell^\infty(G,M)}$ and ${\xi,\eta\in\hh}$, where $\IE$ is a fixed invariant mean on $\ell^\infty(G)$. We shall assume that these properties hold and simply write $\IE^M=\IE$. 
\end{remark}

\begin{proof}[Proof of Prop.\ \ref{P - semi-unitarizable}]
If $G$ is amenable and $\IE$ is a left invariant mean on $\ell^\infty(G,M)$, then $\ip{\p ,\psi}:=\IE_x \p (x)^*\psi(x)$ is a continuous non-degenerate positive $M$-valued semi-definite sesquilinear form and 
\[
\ip{\lambda_x\p ,\ \lambda_x\psi}=\IE_x {\p (x^{-1}y)}^*\psi(x^{-1}y) = \ip{\p ,\psi}
\]
for every $\p ,\psi\in \ell^\infty(G)$, $x\in G$.

Conversely, if $\ip{\,,\,}\colon \ell^\infty(G,M)\times \ell^\infty(G,M)\to M$ is a non-degenerate positive semi-definite sesquilinear form on $\ell^\infty(G,M)$, then the  linear map
\[
\IE(\p ):=\ip{\1,\p }
\]
is left invariant on $\ell^\infty(G,M)$, since
\[
\IE(\lambda_x\p )=\ip{\1,\lambda_x\p }=\ip{\lambda_x\1,\lambda_x\p }=\IE(\p )
\]
for every $\p \in \ell^\infty(G,M)$, $x\in G$. Since $\IE$ is  bounded and $\IE(\1)\neq 0$, it follows by Lemma \ref{L - mean-ext} that $G$ is amenable. 
\end{proof}

\section{semi-definite sesquilinear forms}

Let $G$ be a group,  $\hh$ be a Hilbert space, and $V$ be a complex vector space.

\begin{lemma}\label{pssf-pos-def-unitary-lemma} Let ${X\subset V}$ be a subset stable under multiplication by $i$, and $\ip{\cdot,\cdot}$ be an ${B(\hh)}$-valued positive semi-definite sesquilinear form on $V$. Let ${u\in V}$ be a vector. If $\pi$ is a linear representation of $G$ on $V$ such that the map ${\ip{\pi(\,\cdot\,)\,(u+v),u+v}}$ is positive definite from $G$ to $B(\hh)$ for all ${v\in X}$, then 
\[
\ip{\pi(s)\,u,v}=\ip{u,\pi(s^{-1})\,v}
\]
for all ${v\in X}$ and ${s\in G}$. 
\end{lemma}

\begin{proof}
As ${\ip{\pi(\,\cdot\,)\,(u+v),u+v}}$ is positive definite for all ${v\in X}$, it follows that 
\[
{\ip{\pi(s)\,(u+v),u+v}=\ip{\pi(s^{-1})\,(u+v),u+v}^*=\ip{u+v,\pi(s^{-1})\,(u+v)}}
\]
for all ${v\in E}$ and ${s\in G}$; as $E$ is stable under multiplication by $i$, this implies that 
\begin{align*}
\ip{\pi(s)\,u,v}&=\textstyle\frac{1}{4}\,(\ip{\pi(s)\,(u+v),u+v}-\ip{\pi(s)\,(u-v),u-v}
\\&\ \ \ \ -i\ip{\pi(s)\,(u+iv),u+iv}+i\,\ip{\pi(s)\,(u-iv),u-iv})
\\&=\textstyle\frac{1}{4}\,(\ip{u+v,\pi(s^{-1})\,(u+v)}-\ip{u-v,\pi(s^{-1})\,(u-v)}
\\&\ \ \ \ -i\,\ip{u+iv,\pi(s^{-1})\,u+iv}+i\,\ip{u-iv,\pi(s^{-1})\,u-iv})
\\&=\ip{u,\pi(s^{-1})\,v}
\end{align*}
for all ${v\in E}$ and ${s\in G}$. 
\end{proof}

\begin{lemma}\label{pssf-pos-def-unitary-equiv} Let $\pi$ be a linear representation of $G$ on $V$ and ${\ip{\cdot,\cdot}}$ be a ${B(\hh)}$-valued positive semi-definite left-sesquilinear form on $V$. If the map ${\ip{\pi(\,\cdot\,)\,u,u}}$ from $G$ to $B(\hh)$ is positive definite  for all ${u\in V}$, then 
\[
\ip{\pi(s)\,u,v}=\ip{u,\pi(s^{-1})\,v}
\]
for all ${u,v\in X}$ and ${s\in G}$. 
\end{lemma}

\begin{proof}
It suffices to apply Lemma \ref{pssf-pos-def-unitary-lemma} with ${X=V}$. 
\end{proof}

\begin{proposition}\label{P - Amenable positive semi-definite}
Let $G$ be a  group and $M$ be a von Neumann algebra. The following are equivalent: 
\begin{enumerate}
\item $G$ is amenable;
\item  there exists an $ M$-valued positive semi-definite left-sesquilinear form ${\ip{\cdot,\cdot}}$ on ${\ell^\infty(G, M)}$ and ${\kappa>0}$ such that 
\begin{enumerate}
\item ${\ip{1_ M,1_ M}\neq0}$
\item ${\|{\ip{\p,\p}}\|\leq\kappa\,\|{\p}\|^2}$ for all ${\p\in\ell^\infty(G, M)}$
\item ${\ip{\lambda(\,\cdot\,)\,\p,\p}}$ is positive definite for all ${\p\in\ell^\infty(G, M)}$
\end{enumerate}
\end{enumerate}
\end{proposition}

\begin{proof}
Suppose that $G$ is amenable. Then  
\[
\ip{\p ,\psi}:=\IE_x \p (x)^*\psi(x)
\] 
is an invariant continuous non-degenerate positive $M$-valued semi-definite sesquilinear on $\ell^\infty(G,M)$. As
\[
\ip{\lambda(x)\,\p,\p}=\IE_{y}(\lambda(x)\,\p)(y)^*\p(y)=\IE_{y}\p(x^{-1}y)^*\,\p(y)=\IE_{y}\p(y)^*\,\p(xy)
\]
for all ${\p\in\ell^\infty(G,\M)}$ and ${x\in G}$, it follows by the proof of Prop.\ 2.2 in 
\cite{dcot19}  that  $\ip{\lambda(\cdot )\,\p,\p}$ is positive definite  for all ${\p\in\ell^\infty(G,\M)}$.

The converse follows by Lemma \ref{pssf-pos-def-unitary-equiv} and Prop.\ \ref{P - semi-unitarizable}; we shall provide the details for a strengthening of this implication in Prop. \ref{P - Amenable positive semi-definite invertible} below.
\end{proof}

\begin{remark}
In connection with this, an equivalence between the invariance of an $M$-valued positive semi-definite sesquilinear form and positive definiteness of the corresponding $M$-valued coefficients can be established for  general representations; the proper context for such a statement is that of pre-Hilbert $C^*$-modules.  On the other hand, it is well--known  that the coefficients of a unitary representation of a group $G$ (in the case of trivial coefficients, i.e., $M=\CI$) are positive definite functions on $G$ (see \cite[Prop.\ C.4.3]{bhv}). 
\end{remark}

\begin{corollary}
Let $G$ be a group. The following are equivalent: 
\begin{enumerate}
\item $G$ is amenable;
\item there exists a positive semi-definite sesquilinear form ${\ip{\cdot,\cdot}}$ on ${\ell^\infty(G)}$ and ${\kappa>0}$ such that 
\begin{enumerate}
\item ${\ip{\1,\1}\neq0}$;
\item ${|{\ip{f,f}}|\leq\kappa\,\|{f}\|_\infty^2}$ for all ${f\in\ell^\infty(G)}$;
\item ${\ip{\lambda(\,\cdot\,)\,f,f}}$ is positive definite for all ${f\in\ell^\infty(G)}$.
\end{enumerate}
\end{enumerate}
\end{corollary}

\section{Remarks on Prop.\ \ref{P - semi-unitarizable} and Prop.\ \ref{P - Amenable positive semi-definite}}

As mentioned in the introduction, one difference between the stability condition and the two characterizations of amenability in Prop.\ \ref{P - semi-unitarizable} and Prop.\ \ref{P - Amenable positive semi-definite}, is  that in the latter cases amenability can be used to transform an arbitrary uniformly bounded map (which needs not be an $\e$-representation for small $\e$) into a positive definite map. 
 If the set of $\e$-representations of a group $G$ is too ``thin'', one may not be able to reconstruct an invariant mean from the equivalent stability conditions stated in Prop.\ \ref{stability-equiv}.  

On the other hand, the full set $\ell^\infty(G,M)$ of maps is not required to prove the two propositions. For example, it is sufficient to consider maps $\p\colon G \to \GL_1(M)$, where $\GL_1(M)$ denotes the set of invertible operators, to establish Prop.\ \ref{P - semi-unitarizable}: 

\begin{proposition}\label{P - Amenable positive semi-definite invertible - 1}
Let $G$ be a  group and $M$ be a von Neumann algebra. The following are equivalent: 
\begin{enumerate}
\item $G$ is amenable;
 \item  
 there exists a continuous nondegenerate $M$-valued 
positive semi-definite sesquilinear form $\ip{\,,\,}\colon \ell^\infty(G,M)\times \ell^\infty(G,M)\to M$ such that 
\[
\ip{\lambda_x \p,\lambda_x\psi}=\ip{\p,\psi}
\]
for all $x\in G$ and all $\p,\psi\in \ell^\infty(G,\GL_1(M))$;
\end{enumerate}
\end{proposition}

\begin{proof}
The direct implication follows by Prop.\ \ref{P - semi-unitarizable}. Conversely, using the notation of Prop.\ \ref{P - semi-unitarizable}, the  restriction of the linear map $\IE(\p ):=\ip{\1,\p }$ to $\ell^\infty(G)$ 
satisfies $\IE(\lambda_x\p )=\IE(\p )$
for every  $x\in G$ and every $\p \in \ell^\infty(G)$ such that $\p(y)\neq 0$ for every $y\in G$. Since $\IE$ is  bounded, this equality holds for every $\p \in \ell^\infty(G)$,  and since $\IE(\1)\neq 0$, it follows  that $G$ is amenable.
\end{proof}

We shall require the following strengthening of  Prop.\ \ref{P - Amenable positive semi-definite} in the proof of  Prop.\ \ref{P - nonamenable intro} and Prop.\ \ref{amenable-crit}   (see   \ts\ref{S - amenable 2} and \ts\ref{S - restricting to epsilon rep}). 

\begin{proposition}\label{P - Amenable positive semi-definite invertible}
Let $G$ be a  group and $M$ be a von Neumann algebra. The following are equivalent: 
\begin{enumerate}
\item $G$ is amenable;

\item   there exists a continuous nondegenerate $M$-valued 
positive semi-definite sesquilinear form $\ip{\,,\,}\colon \ell^\infty(G,M)\times \ell^\infty(G,M)\to M$ such that 
 ${\ip{\lambda(\,\cdot\,)\,\p,\p}}$ is positive definite for all ${\p\in\ell^\infty(G, \GL_1(M))}$.
\end{enumerate}
\end{proposition}

It is clear that this proposition holds identically with $\lambda(\cdot)$ replaced by $\lambda(\cdot^{-1})$.

\begin{proof}
The direct implication follows by Prop.\ \ref{P - Amenable positive semi-definite}. Conversely,   let ${\Psi:\ell^\infty(G,M)\to M}$ be the linear map defined by 
\[
\Psi(\p)=\ip{\1,\p}
\]
Let 
\[
{ X=\{\p\in\ell^\infty(G,M):\|{\p}\|<1\}}
\]
As ${\|{\p(x)}\|<1}$ for all ${\p\in X}$ and ${x\in G}$, in turn ${\1+\p(x)}$ is invertible for all ${\p\in X}$ and ${x\in G}$, and thus ${\ip{\lambda(\,\cdot\,)(\1+\p),\1+\p}}$ is positive definite for all ${\p\in X}$. As $ X$ is stable under multiplication by $i$, it follows by Lemma \ref{pssf-pos-def-unitary-lemma} that 
\[
{\Psi(\lambda(s)\,\p)=\ip{\1,\lambda(s)\,\p}=\ip{\lambda(s^{-1})\,\1,\p}=\ip{\1,f}=\Psi(\p)}
\]
for all ${s\in G}$ and ${\p\in X}$. By linearity, it follows that ${\Psi(\lambda(s)\,\p)=\Psi(\p)}$ for all ${s\in G}$ and ${\p\in\ell^\infty(G,M)}$, and thus $\Psi$ is left invariant on ${\ell^\infty(G,M)}$. Furthermore, it follows by the Cauchy-Schwarz inequality and Prop.\ 2.3 in \cite{paschke73} that 
\[
{|{\Psi(\p)}|\leq\|{\ip{\1,\p}}\|\,\|{\eta}\|^2\leq\|{\ip{\1,\1}}\|^{1/2}\,\|{\ip{\p,\p}\|^{1/2}\leq\kappa\,\|{\p}\|}}
\]
for all ${\p\in\ell^\infty(G,M)}$, and thus $\Psi$ is bounded. As $\Psi(\1)\neq 0$, it follows by  Lemma \ref{L - mean-ext} that $G$ is amenable. 
\end{proof}

\begin{remark}
 The arguments rely on the fact that the set of  maps $\p\colon G\to \GL_1(M)$ is large enough to generate $\ell^\infty(G,M)$, and  they can  be extended to other sets of maps with similar properties (e.g., the set of $M$-valued unitary maps).  
\end{remark}

\section{A positive definite characterization of the invariance of a mean}\label{S - amenable 2} 
The following result characterizes amenability as a property that yields for every uniformly
 bounded map $\p$ into a positive definite map $\psi$ which is close to it if $\p$ is $\e$-multiplicative:

\begin{proposition}\label{amenable-crit}
Let $G$ be a  group and $M$ be a von Neumann algebra. The following are equivalent: 
\begin{enumerate}
\item $G$ is amenable
\item there exists a state ${\tau\in\mathrm{St}(\ell^\infty(G))}$ such that for every every uniformly bounded map ${\p:G\to M}$, there exists a positive definite map ${\psi:G\to M}$ such that 
\[
{\Phi(\p(x)^*\,\psi(x))=\tau_y\,\Phi(\p(x)^*\,\p(xy)\,\p(y)^*)}
\]
for all ${x\in G}$ and every normal linear functional ${\Phi\in M_*}$
\end{enumerate}
\end{proposition}

\begin{proof}
The proof of the implication ${(1)\Rightarrow(2)}$ is standard. Suppose that ${(2)}$ holds and let ${\tau\in\mathrm{St}(\ell^\infty(G))}$ be a state witnessing ${(2)}$. We let ${\Psi_\tau:\ell^\infty(G,M)\to M}$ be the positive bounded linear map defined by 
\[
\Phi(\Psi_\tau(\p))=\tau_x\,\Phi(\p(x))
\]
for ${\Phi\in M_*}$ and ${\p\in\ell^\infty(G,M)}$, where we identify $M$ with ${(M_*)^*}$. 
Let ${\ip{\cdot,\cdot}}$ be the $M$-valued positive semi-definite left-sesquilinear form on ${\ell^\infty(G,M)}$ defined by 
\[
\ip{\p,\psi}=\Psi_\tau(\p^*\psi).
\]
By Prop.\ \ref{P - Amenable positive semi-definite invertible}, it suffices to show that  ${\ip{\lambda(\,\cdot\,^{-1})\,\p,\p}}$ is a positive definite map on $G$ for all ${\p\in\ell^\infty(G, \GL_1(M))}$.

Let ${\p\in\ell^\infty(G, \GL_1(M))}$. By assumption, there exists a positive definite map ${\psi:G\to M}$ such that 
\[
{\Phi(\p(x)\,\psi(x))=\Phi((\Psi_\tau)_y\p(x)\,\p(xy)^*\,\p(y))}=\Phi(\p(x)(\Psi_\tau)_y\p(xy)^*\,\p(y))
\]
for all ${x\in G}$ and every normal linear functional ${\Phi\in M_*}$. Since $\p(x)$ is invertible for all $x\in G$, it follows that
\[
\psi(x^{-1})=(\Psi_\tau)_y\,\p(x^{-1}y)^*\,\p(y)={\ip{\lambda(\,\cdot\,)\,\p,\p}}.
\]
This shows that ${\ip{\lambda(\,\cdot\,^{-1})\,\p,\p}}$ is positive definite.
\end{proof}

\begin{remark}\label{R - p close to psi}
Suppose  ${\tau\in\mathrm{St}(\ell^\infty(G))}$ is a state and $\p\colon G\to U(\hh)$ is a unitary map, $\psi\colon G\to B(\hh)$ is a uniformly bounded map such that the identity
\[
{\Phi(\p(x)^*\,\psi(x))=\tau_y\,\Phi(\p(x)^*\,\p(xy)\,\p(y)^*)}
\]
is verified for all ${x\in G}$ and every normal linear functional ${\Phi\in B(\hh)^*}$; then 
\[
\|{\p(x)-\psi(x)}\|\leq \sup_{y\in G}\|{\p(x)\,\p(y)-\p(xy)}\| 
\]
for all $x\in G$. In particular, if $\p$ is an $\e$-representation then 
\[
\|\p-\psi\|\leq \e.
\]
\end{remark}

\begin{proof} We have
\[
{|{\Phi(\Id_\hh-\p(x)^*\,\psi(x))}|=|{\tau_y\,\Phi(\Id_\hh-\p(x)^*\,\p(xy)\,\p(y)^*)}}|
\]
for all ${x\in G}$ and every normal linear functional ${\Phi\in B(\hh)^*}$, and thus 
\begin{align*}
\|{\p(x)-\psi(x)}\|&=\|{\Id_\hh-\p(x)^*\,\psi(x)}\|\\
&=\sup\{|{\Phi(\Id_\hh-\p(x)^*\,\psi(x))|}:\Phi\in B(\hh)_*,\,\|{\Phi\|}\leq1\}\\
&=\sup\{|{\tau_y\,\Phi(\Id_\hh-\p(x)^*\,\p(xy)\,\p(y)^*)|}:\Phi\in B(\hh)_*,\,\|{\Phi}\|\leq1\}\\
&\leq\sup\{\|{\Id_\hh-\p(x)^*\,\p(xy)\,\p(y)^*}\|:y\in G\}\\
&=\sup\{\|{\p(x)\,\p(y)-\p(xy)}\|:y\in G\}
\end{align*}
for all ${x\in G}$. 
\end{proof}

\begin{question}
Does there exist a non-amenable group $G$ such that  for every von Neumann algebra $M$ and every uniformly bounded map ${\p:G\to U(M)}$, there exists a positive definite map ${\psi:G\to M}$ such that 
\[
\|{\p(x)-\psi(x)}\|\leq \sup_{y\in G}\|{\p(x)\,\p(y)-\p(xy)}\| 
\]
for all $x\in G$? 
\end{question}

\section{Restricting to $\e$-representations}\label{S - restricting to epsilon rep}

We now consider the situation where we restrict the set of maps $\p\colon G\to B(\hh)$ for which a positive definite map $\psi$ such that
\[
{\Phi(\p(x)^*\,\psi(x))=\tau_y\,\Phi(\p(x)^*\,\p(xy)\,\p(y)^*)}\ \ \forall x\in G, \ \forall \Phi\in B(\hh)_*
\]
must exist to the set of $\delta$-representations $\p\colon G\to U(\hh)$ for some $\delta>0$.

\begin{proposition} 
Let $G$ be a group and $\hh$ be a Hilbert space.
Suppose that there exists a state ${\tau\in\mathrm{St}(\ell^\infty(G))}$ and $0<\delta<1$ such that for every   $\delta$-representation ${\p:G\to U(\hh)}$, there exists a positive definite map ${\psi:G\to B(\hh)}$ such that 
\[
{\Phi(\p(x)^*\,\psi(x))=\tau_y\,\Phi(\p(x)^*\,\p(xy)\,\p(y)^*)}
\]
for all ${x\in G}$ and every normal linear functional ${\Phi\in B(\hh)_*}$. Then $G$ is stable.
\end{proposition}

\begin{proof}
We establish (4) in Prop.\ \ref{stability-equiv}. We fix $\tau\in\mathrm{St}(\ell^\infty(G))$ and $0<\delta<1$ witnessing the assumptions of the proposition. Let $0<\e<\delta$, let   ${\p:G\to U(\hh)}$ be an   $\e$-representation, and let $\psi$ be a positive definite map such that \[
{\Phi(\p(x)^*\,\psi(x))=\tau_y\,\Phi(\p(x)^*\,\p(xy)\,\p(y)^*)}
\]
for all ${x\in G}$ and every normal linear functional ${\Phi\in B(\hh)_*}$. Since $\p\in \ell^\infty(G,\GL_1(M))$, the proof of Prop.\ \ref{amenable-crit} show that $\psi$ is uniquely determined by the formula
\[
\psi(x)=(\Psi_\tau)_y\,\p(xy)\,\p(y)^*
\]
for all $x\in G$. Since $\p$ is an $\e$-representation, Remark \ref{R - p close to psi} shows that 
\[
\|\p-\psi\|\leq \e.
\]

We show that $\psi$ is $2\e^2$-unitary, following the proof of Kazhdan's theorem given in \cite{bot13}. Thus, (4) in Prop.\ \ref{stability-equiv} holds with $\kappa_1=2$, $\kappa_2=1$, $p=2$.

Since

\begin{align*}
(\Psi_\tau)_y(\p(xy)&-\p(x)\,\p(y))(\p(xy)-\p(x)\,\p(y))^*\\
&=(\Psi_\tau)_y\p(xy)\,\p(xy)^*-\p(x)\,\psi(x)^*-\psi(x)\,\p(x)^*+\p(x)\,\p(x)^*
\\&=\Id_\hh-\psi(x)\,\psi(x)^*+(\p(x)-\psi(x))(\p(x)-\psi(x))^*
\end{align*}
it follows that 
\begin{align*}
 \|{\Id_\hh-\psi(x)\,\psi(x)^*}\|&\leq \|{(\p(x)-\psi(x))(\p(x)-\psi(x))^*}\|\\
&\hspace{1cm} + \|{(\Psi_\tau)_y(\p(xy)-\p(x)\,\p(y))(\p(xy)-\p(x)\,\p(y))^*}\|
\\&\leq \|{(\p(x)-\psi(x))(\p(x)-\psi(x))^*}\|\\
&\hspace{1cm} +(\Psi_\tau)_y \|{(\p(xy)-\p(x)\,\p(y))(\p(xy)-\p(x)\,\p(y))^*}\|
\\&\leq \|{\p(x)-\psi(x)}\| \|{(\p(x)-\psi(x))^*}\|\\
&\hspace{1cm} +(\Psi_\tau)_y \|{\p(xy)-\p(x)\,\p(y)}\,\| \|{(\p(xy)-\p(x)\,\p(y))^*}\|
\\&= \|{\p(x)-\psi(x)}\|^2+(\Psi_\tau)_y \|{\p(xy)-\p(x)\,\p(y)}\|^2\\
&<\e^2+\e^2=2\e^2
\end{align*}
for all ${x\in G}$.
\end{proof}

\begin{remark}
One can in fact establish (4) in Prop.\ \ref{stability-equiv}  with $\kappa_1=1$, $\kappa_2=1$, $p=2$. Indeed, in the notation of the proof, 
as ${\psi(x)\,\psi(x)^*}$ is positive and 
\[
{\|{\psi(x)\,\psi(x)^*}\|\leq\|{\psi}\|^2=\|{\psi(e)}\|=1}
\]
for all ${x\in G}$, it follows that ${\Id_\hh-\psi(x)\,\psi(x)^*}$ is positive for all ${x\in G}$. Furthermore, ${|({\p(x)-\psi(x))^*}|^2}$ is also positive for all ${x\in G}$ and the equality 
\[
{\Id_\hh-\psi(x)\,\psi(x)^*+|{(\p(x)-\psi(x))^*}|^2=(\Psi_\tau)_y\,|{(\p(xy)-\p(x)\,\p(y))^*}|^2} 
\]
for all ${x\in G}$ shows that 
\[
{0\leq \Id_\hh-\psi(x)\,\psi(x)^*\leq(\Psi_\tau)_y\,|{(\p(xy)-\p(x)\,\p(y))^*}|^2}
\]
for all ${x\in G}$. This then implies that 
\begin{align*}
\|{\Id_\hh-\psi(x)\,\psi(x)^*}\|&\leq\|{(\Psi_\tau)_y\,|{(\p(xy)-\p(x)\,\p(y))^*}|^2}\|
\\&\leq\sup\{\|{|{(\p(xy)-\p(x)\,\p(y))^*}|^2}:y\in G\}\|
\\&=\sup\{\|{(\p(xy)-\p(x)\,\p(y))^*}\|^2:y\in G\}
\\&=\sup\{\|{\p(xy)-\p(x)\,\p(y)}\|^2:y\in G\}\leq\e^2
\end{align*}
for all ${x\in G}$.
\end{remark}

In \cite{kazhdan82} Kazhdan proves that the compact surface groups of genus $\geq 2$ are not stable. 
More recently, an interesting construction of Rolli \cite{Rolli} gives for every finite dimensional Hilbert space $\hh$ and every $\e>0$ an $\e$-representation  $F_n\to U(\hh)$, where $F_n$ is the free group on $n\geq 2$ generators, which is not close to a unitary representation.  By using the results of Burger, Ozawa and Thom \cite{bot13}, we can obtain the following result for groups that contain $F_2$.

\begin{corollary}
Suppose that $G$ contains a non-abelian free group. For every state ${\tau\in\St(\ell^\infty(G))}$, Hilbert space $\hh$, and small enough $\e>0$, there exists an $\e$-representation ${\p: G\to U(\hh)}$ such that for every positive definite map ${\psi: G\to B(\hh)}$, there exists some normal linear functional ${\Phi\in B(\hh)^*}$ such that 
\[
{\Phi(\p(x)^*\,\psi(x))\neq\tau_y\,\Phi(\p(x)^*\,\p(xy)\,\p(y)^*)}
\]
for some ${x\in G}$.  
\end{corollary}

This follows by Cor.\ 3.8 in \cite{bot13}. In fact, it is sufficient to suppose that the comparison map
\[
H^2_b(G,\IR)\to H^2(G,\IR)
\]
from bounded 2-cohomology to usual 2-cohomology is not injective, by Cor.\ 3.5 in \cite{bot13}.

\section{Remarks on unitarily invariant norms}

Finally, we note that the equality
\[
{\Phi(\p(x)^*\,\psi(x))=\tau_y\,\Phi(\p(x)^*\,\p(xy)\,\p(y)^*)} \ \forall x\in G,\ \forall \Phi\in M_*
\]
implies a norm estimate on the distance between $\p$ and $\psi$, for every unitarily invariant norm on a finite  factor $M$.

\begin{proposition}\label{P - unitarily invariant norms}
Let $G$ be a countable group, $M$ be a finite factor, ${\tau\in\mathrm{St}(\ell^\infty(G))}$ be a state, ${\p,\psi:G\to M}$ be uniformly bounded maps such that 
\[
{\Phi(\p(x)^*\,\psi(x))=\tau_y\,\Phi(\p(x)^*\,\p(xy)\,\p(y)^*)}
\]
for all ${x\in G}$ and every normal linear functional ${\Phi\in M_*}$, and $\|\cdot\|$ be an ultraweakly lower semi-continuous unitarily invariant norm on $M$. Then 
\[
{\|{\p(x)-\psi(x)}\|\leq\tau_y\,\|{\p(xy)-\p(x)\,\p(y)}}\|
\]
for all ${x\in G}$. 
\end{proposition}

\begin{proof}
It follows by Theorem C in \cite{fhns08} that 
\begin{align*}
\|{\p(x)-\psi(x)}\|&=\|{\Id_\hh-\p(x)^*\,\psi(x)}\|
\\&=\sup\{|{\Phi(\Id_\hh-\p(x)^*\,\psi(x))}|:\Phi\in M_*,\,\|{X_\Phi}\|_*\leq1\}
\\&=\sup\{|{\tau_y\,\Phi(\Id_\hh-\p(x)^*\,\p(xy)\,\p(y)^*)|}:\Phi\in M_*,\,\|{X_\Phi}\|_*\leq1\}
\\&=\sup\{|{\Phi(\Id_\hh-(\Psi_\tau)_y\,\p(x)^*\,\p(xy)\,\p(y)^*)}|:\Phi\in M_*,\,\|{X_\Phi}\|_*\leq1\}
\\&=\|{\Id_\hh-(\Psi_\tau)_y\,\p(x)^*\,\p(xy)\,\p(y)^*}\|
\\&\leq\tau_y\,\|{\Id_\hh-\p(x)^*\,\p(xy)\,\p(y)^*}\|
\\&=\tau_y\,\|{\p(x)\,\p(y)-\p(xy)}\|
\end{align*}
for all ${x\in G}$, where ${\|\cdot\|_*}$ denotes the dual norm of $\|\cdot \|$ and ${X_\Phi}$ denotes the bounded operator in $M$ giving rise to the functional $\Phi$. 
\end{proof}

\begin{remark}
The assumption that $\|\cdot\|$ be ultraweakly lower semi-continuous is redundant, since Theorem C in \cite{fhns08} implies that every unitarily invariant norm on a finite factor is ultraweakly lower semi-continuous. 
\end{remark}

\end{document}